\documentclass[12pt]{amsart}
\usepackage[top=30truemm,bottom=30truemm,left=30truemm,right=30truemm]{geometry}
\UseRawInputEncoding
\usepackage{mathrsfs}
\usepackage{bm}
\usepackage{amsfonts,amssymb}
\usepackage{dsfont}
\usepackage{extarrows}
\usepackage{amsmath}
\usepackage{mathrsfs}
\usepackage{enumerate}
\usepackage{amscd}
\usepackage[all,2cell]{xy}
\usepackage{hyperref}

\theoremstyle{plain}

\theoremstyle{definition} %text of this environment is typesetted in roman letters

\newtheorem{thm}{Theorem}[section]
\newtheorem{cor}[thm]{Corollary}
\newtheorem{lem}[thm]{Lemma}
\newtheorem{prop}[thm]{Proposition}

\theoremstyle{definition}
\newtheorem{defn}{Definition}[section]

\theoremstyle{remark}
\newtheorem{rem}{\bf Remark}[section]
\newtheorem{exm}{\bf Example}[section]

\newcommand{\be}{\begin{equation}}
\newcommand{\ee}{\end{equation}}
\newcommand{\bea}{\begin{eqnarray}}
\newcommand{\eea}{\end{eqnarray}}
\newcommand{\ben}{\begin{eqnarray*}}
	\newcommand{\een}{\end{eqnarray*}}
\newcommand{\bt}{\begin{split}}
	\newcommand{\et}{\end{split}}
\newcommand{\bet}{\begin{equation}}
\newcommand{\mc}{\mathbb{C}}
\newcommand{\mr}{\mathbb{R}}
\newcommand{\ra}{\rightarrow}
\newcommand{\beq}{\begin{equation*}}
\newcommand{\eeq}{\end{equation*}}
\newcommand{\bal}{\begin{aligned}}
\newcommand{\eal}{\end{aligned}}

\newcommand{\ddbar}{\partial \bar{\partial}}
\newcommand{\dbar}{\bar{\partial}}
\newcommand{\pa}{\partial}

\newcommand{\calO}{{\mathcal{O}}}%
\newcommand{\loc}{{\textup{loc}}}%
\newcommand{\Ker}{{\textup{Ker}}}
\newcommand{\Img}{{\textup{Im}}}

\newcommand{\Dom}{{\textup{Dom}}}

\newcommand{\pd}{\partial}%
\renewcommand{\leq}{\leqslant}%
\renewcommand{\geq}{\geqslant}%
\newcommand{\inner}[1]{\langle#1\rangle}
\newcommand{\iinner}[1]{\langle\langle#1\rangle\rangle}

\makeatletter
\g@addto@macro{\endabstract}{\@setabstract}
\newcommand{\authorfootnotes}{\renewcommand\thefootnote{\@fnsymbol\c@footnote}}%
\makeatother

\begin{document}
\title[$L^2$-optimal]{A new characterization of $L^2$-domains of holomorphy with null thin complements via $L^2$-optimal conditions}

\author[Z. Liu]{Zhuo Liu}
\address{Beijing Institute of Mathematical Sciences and Applications; Department of Mathematics and Yau Mathematical Sciences Center, Tsinghua University, Beijing 101408, China.}
\email{liuzhuo@amss.ac.cn}

\author[X.Zhang]{Xujun Zhang}
\address{Institute of Mathematics, Academy of Mathematics and Systems Science, Chinese Academy of Sciences, Beijing, 100190, P. R. China}
\email{xujunzhang@amss.ac.cn}



%%%%%%%%%%%%%%%%%%%%%%%%%%%%%%%%%%%%%%%%%%%%%%%%%%%%%%

%%%%%%%%%%%%%%%%%%%%%%%%%%%%%%%%%%%%%%%%%%%%%%%%%%%%%%

\maketitle

\begin{abstract}
In this paper, 
we show that the $L^2$-optimal condition implies the $L^2$-divisibility of $L^2$-integrable holomorphic functions. 
As an application, we offer a new characterization of bounded $L^2$-domains of holomorphy with null thin complements using the $L^2$-optimal condition, 
which appears to be advantageous in addressing a problem proposed by Deng-Ning-Wang.
Through this characterization,
we show that a domain in a Stein manifold with a null thin complement, admitting an exhaustion of complete K\"ahler domains, remains Stein.
By the way, we construct an $L^2$-optimal domain that does not admit any complete K\"ahler metric. 
\end{abstract}

%%%%%%%%%%%%%%%%%%%%%%%%%%%%%%%%%%%%%%%%%%%%%%%%%%%%%%
\section{Introduction}

The classical Levi Problem in several complex variables queries whether every domain in $\mathbb{C}^n$ with a $C^2$-smooth pseudoconvex boundary is indeed a domain of holomorphy. This was affirmatively resolved through the pioneering efforts of Oka (\cite{Oka}), Bremermann (\cite{Bremermann})
and Norguet (\cite{Norguet}). Furthermore, the work of Andreotti-Vesentini (\cite{Andreotti-Vesentini}),
Kohn (\cite{Kohn63}) and
H\"ormander (\cite{Hor65}) demonstrated that this problem can also be addressed via the $L^2$-existence theorem for the $\bar{\partial}$-equation.  
Concepts like plurisubharmonicity and Griffiths/Nakano positivity are crucial in the $L^2$ theory within several complex variables and complex geometry, leading to many significant discoveries. Plurisubharmonic functions, which are not necessarily smooth, offer substantial advantages in various problems. As for vector bundles, there is a continuous interest in investigating singular metrics with certain types of positivity.

It is well-known that a smooth Hermitian metric $h$ on a holomorphic vector bundle is Griffiths semi-positive if and only if $\log|u|_{h^*}$ is psh for any local holomorphic section $u$ of the dual bundle. This characterization naturally leads to a definition of Griffiths positivity for singular Hermitian metrics (see \cite{BP08,Rau15,PT18}), which has proven to be very useful. In \cite{DWZZ,DNW21,DNWZ22}, Deng, Ning, Wang, Zhang, Zhou established the converse $L^2$ theory by giving alternative characterizations of plurisubharmonicity and Griffiths/Nakano positivity in terms of various $L^2$-conditions for $\bar{\pd}$.
They proved that a smooth Hermitian metric is Nakano semi-positive if and only if it satisfies the ``\textit{optimal $L^2$-estimate condition}''. This characterization leads to a definition of Nakano positivity for singular Hermitian metrics (see \cite{In20}) and provides a positive answer to a question of Lempert (see \cite{LYZ21}).
Moreover, Deng-Zhang (\cite{Deng-Zhang}) gave  new characterizations of bounded pseudoconvex domains and convex domains with $C^2$-smooth boundaries.

Since Nakano positivity is stronger than Griffiths positivity, it is logical to question whether the optimal $L^2$-estimate condition implies Griffiths positivity for singular Hermitian metrics under certain regularity conditions. Specifically, when dealing with a line bundle, this question goes back to an open problem from \cite{DNW21}: "$\cdots\cdots$ it is natural to ask whether the regularity condition on $\psi$ in  Theorem 1.2 can be weakened to being continuous?"

Recall that a singular Finsler metric is Griffiths semipositive if and only if the zero section of $E^*$ admits a pseudoconvex neighborhood (\cite[Proposition 3.6.5]{kobayashi}). When $E$ is a line bundle, this pseudoconvex neighborhood can be represented as the Hartogs domain determined by the weight function over $E$. Specifically, for an upper semi-continuous function $\psi$ defined on a pseudoconvex domain $D\subset\mathbb{C}^n$, $\psi$ is plurisubharmonic
 if and only if the Hartogs domain $$D_{\psi}:=\{D \times \mc=(z,w):z \in D, |w|<e^{-\psi(z)}\}$$ is pseudoconvex (\cite[Chapter I, Example 7.4. b]{Demaily}).
It is evident that, in general, $D_{\psi}$ does not possess a $C^2$-smooth boundary. Therefore, it is meaningful to explore under what kind of regularity condition for the boundary \textit {the $L^2$-optimal domain} becomes a pseudoconvex domain.

In this note, we will demonstrate that the $L^2$-optimal condition implies the $L^2$-divisibility of $L^2$-integrable holomorphic functions. As an application, we show that any bounded $L^2$-optimal domain with a null thin complement is an $L^2$-domain of holomorphy.

Now let us recall definitions of the $L^2$-optimal condition and $L^2$-domains of holomorphy.

\begin{defn}\label{def:L^2 peroperty}
Let $D$ be a domain in $\mc^n$ and
$\psi$  a real-valued continuous function defined on $ D$.
\begin{enumerate}[(i)]
\item We say that the pair $(D,\psi)$ is \emph{$L^2$-optimal},
if for any smooth strictly plurisubharmonic function $\phi$ on $\mc^n$,
the equation $\bar\partial u=f$ can be solved on $D$ for any $\dbar$-closed $(0,1)$-form
$f \in L^2_{(0,1)}(D,\phi+\psi)$
with the estimate:
$$
\int_D|u|^2 e^{-\phi-\psi} d\lambda
\leq
\int_D \langle B_\phi^{-1} f, f\rangle e^{-\phi-\psi} d\lambda,
$$
where
$B_\phi f = \sum^{n}_{i,j=1}  \frac{\partial^2\phi}{\partial z_i \partial\bar z_ j} f_j d\bar z_i$ and $d\lambda$ is the Lebesgue measure in $\mc^n$.

\item In particular,
we say that $D$ is \emph{$L^2$-optimal} if  the pair $(D,0)$ is \emph{$L^2$-optimal}.
\end{enumerate}
\end{defn}

It follows from the $L^2$-existence theorem obtained by H\"ormander \cite{Hor65}
and Demailly \cite{Dem-82} that $(D,\psi)$ is $L^2$-optimal if $D$ is complete K\"ahler and $\psi$ is plurisubharmonic.
But the converse maybe not true. In fact, we can construct an $L^2$-optimal domain which does not admit any complete K\"ahler metric (See Example \ref{exa}).

A classical problem in several complex variables is the division problem for holomorphic functions.
In \cite{Skoda}, Skoda found a celebrated $L^2$-division theorem for holomorphic functions on pseudoconvex domains. Subsequently, there have been various generalizations of Skoda's division theorem, as seen in \cite{Skoda-2, Oh-02, Varolin}.
Building on the pioneering efforts of Skoda, we prove that the $L^2$-optimal condition implies a weakened version of the Bochner-type inequality (Lemma \ref{lem:weak-basic-inequality}). Fortunately, this is sufficient to establish $L^2$-divisibility for $L^2$-integrable holomorphic functions, which is the first result of this note.

\begin{thm}[Theorem \ref{thm:skoda-division}]\label{thm:1}
Let $D\subset\mc^n$ be a  domain and $\psi$ be an upper semi-continuous function on $D$.
	Assume that $(D,\psi)$ is $L^2$-optimal. Suppose that there exists a $g\in\calO(D)^{\oplus p}$ with $0<\inf_D|g|^2\le\sup_D|g|^2<+\infty$ and $\sup_D|\pa g|^2<+\infty$. Then for any holomorphic function $f \in L^2 (D,\psi)$, there are $p$ holomorphic functions $h_i\in L^2 (D,\psi), i=1,\cdots,p,$
such that
$$
\sum_{i=1}^{p} h_i g_i=f.
$$
\end{thm}

\begin{rem}
  If $D$ is a bounded domain in $\mc^n$, then we can take $g(z)=z-z^0\in\calO(D)^{\oplus n}$ with $z^0\notin\overline{D}$. If $D\in\mc^n$ is a tube domain associated to a bounded domain $\omega\in \mr^n$, which means that $$D=\{z\in\mc^n:\Re z\in\omega\},$$ then we can take $g(z)=e^{z-z^0}-1\in\calO(D)^{\oplus n}$ with $z^0\notin\overline{D}$.
\end{rem}

\begin{defn}[\cite{PW}]
	Let $D$ be a domain in $\mc^n$.
	We say that $D$ is an $L^2$-domain of holomorphy if for any pair of open  sets $U_1, U_2\; (\subset \mc^n)$ with $\emptyset \neq U_1 \subset D \cap U_2 \neq U_2$,
	$U_2$ connected,
	there exists $f \in L^2 (D)$ such that for any $F \in \calO(U_2)$,
	we have
	$$
	f \big|_{U_1} \neq F\big|_{U_1}.
	$$
\end{defn}

It is obvious that $L^2$-domains of holomorphy are  domains of holomorphy and hence pseudoconvex. Then by  the $L^2$-existence theorem,  $L^2$-domains of holomorphy is $L^2$-optiaml. Conversely,
inspired by the work of Diederich and Pflug\cite{Diederich-Pflug},
if $D$ is a bounded domain or a tube domain associated to a bounded domain in $\mr^n$  with a \textit{null thin complement},
which means that $\mathring{\overline {D}} \backslash D=\emptyset$,
then Theorem \ref{thm:1} allows us to construct $L^2$-integrable holomorphic functions that cannot be analytically extended across the boundary. Consequently, we attain a  new characterization of $L^2$-domains of holomorphy.

\begin{thm}\label{thm:main}Let $D\in\mc^n$ be a bounded domain or a tube domain associated to a bounded domain in $\mr^n$ with a null thin complement. Then  $D$ is $L^2$-optimal if and only if $D$ is an $L^2$-domain of holomorphy.
\end{thm}

\begin{rem}
	\begin{enumerate}
	\item In \cite{Grauert}, Grauert showed that
any bounded complete K\"ahler domain in $\mc^n$ with a real analytic smooth boundary is a domain of holomorphy.
And Ohsawa showed that this also holds when the boundary is $C^1$-smooth in \cite{Oh-80}.
	In \cite{Diederich-Pflug}, Diederich-Pflug proved that any complete K\"ahler domain $D$ with a null thin completement is an $L^2$-domain of holomorphy.
		\item The proof of Theorem \ref{thm:main} only involves strictly plurisubharmonic functions $\varphi(z)$ of the form $a\|z \|^2$ on $\mc^n$ with $a>0$.
\item Compared to the work of Laufer (\cite{Laufer}) and Laurent-Thi\'{e}baut-Shaw (\cite{Shaw}), among others, we avoid the use of cohomological methods and require only first-order solvability with an optimal $L^2$ estimate.
		\item The $L^2$ estimate condition in Definition \ref{def:L^2 peroperty} is essential since we can solve the $\dbar$ equation on the ball minus a small ball (\cite{ShawMeichi}).
	\end{enumerate}
\end{rem}

It's worth noting that if $\mathring{\overline D}=D$ and $\psi$ is continuous on $\overline D$, then the Hartogs domain $D_\psi$ satisfies $\mathring{\overline {D_\psi}}=D_{\psi}$. Consequently, a direct corollary of Theorem \ref{thm:main} is as follows:

\begin{cor}
	Let $D$ be a bounded domain in $\mc^n$ with  a null thin complement and
$\psi$ a real-valued continuous function on $\overline D$.
Then $D_{\psi}$ is $L^2$-optimal in $\mc^{n+1}$
if and only if $D$ is pseudoconvex and $\psi$ is plurisubharmonic.
\end{cor}

	 When $D$ possesses a $C^2$-smooth boundary and $\psi$ is a $C^2$-smooth function defined on $\overline D$, Deng and Zhang \cite{Deng-Zhang} showed that the pair $(D,\psi)$ is $L^2$-optimal if and only if $D$ is pseudoconvex and $\psi$ is plurisubharmonic. Their approach relies on the Morrey trick involving the boundary terms in the Bochner-Hörmander-Kohn-Morrey formula \cite{Hor65}. However, this technique fails when the $C^2$-smooth regularity condition is relaxed. In contrast, our method differs from theirs.

Based on this, a natural question arises: Does the $L^2$-optimal condition of $(D,\psi)$ imply that $D_{\psi}$ is also $L^2$-optimal? If the answer is affirmative, 
then the question posed in \cite{DNW21} can be resolved positively.

The previous discussion on the $L^2$-division for holomorphic functions and the characterization of $L^2$-domains of holomorphy can be generalized naturally for relatively compact $L^2$-optimal domains in Stein manifolds (Theorem \ref{thm:skoda-division'} and Theorem \ref{thm stein}).

In addition,
notice that the limit of an increasing sequence 
of $L^2$-optimal domains remains $L^2$-optimal (see Proposition \ref{prop closedness}), 
then by Theorem \ref{thm:main}, 
we can show that a domain in a Stein manifold, with a null thin complement, admitting an exhaustion of complete K\"ahler domains is also Stein (see Corollary \ref{cor:aaa}).

The rest of this note is organized as follows.
In \S \ref{sec:notations},
we clarify some basic notations and prove an $L^2$-division theorem (Theorem \ref{thm:1}).
In \S \ref{sec:proof of main}, we prove our main Theorem \ref{thm:main} and their generalizations on Stein manifolds.
In \S \ref{sec:Negligible-set}, we show some properties for the $L^2$-optimal condition and construct an $L^2$-optimal domain which does not admit any complete K\"ahler metric.

\textbf{Acknowledgement:}
The authors are very grateful to Professor Fusheng Deng, Professor
Xiangyu Zhou and Doctor Wang Xu for valuable suggestions on related topics.

\section{A division theorem on $L^2$-optimal domains}\label{sec:notations}
Let $D$ be a domain in $\mc^n$.
We denote by $\wedge^{0,q}(D)$ the space of  $(0,q)$-forms on $D$, for any $0\leq q\leq n$ ((0,0)-forms are just smooth functions), and
$\wedge^{0,q}_c(D)$ the elements in
$\wedge^{0,q}(D)$ with compact support.
Let $\varphi$ be a real-valued continuous function on $D$.
Given $\alpha=\sum_I\alpha_Id\bar z_I, \beta=\sum_I\beta_Id\bar z_I\in \wedge^{0,q}(D)$, we define the products of $\alpha$ and $\beta$ and the corresponding norm with respect to $\varphi$ as follows:
$$
\bal
\inner{\alpha, \beta}_{\varphi} &= \sum_I \alpha_I \cdot \overline\beta_I e^{-\varphi},
|\alpha|^2_{\varphi}=\inner{\alpha,\alpha}_{\varphi}. \\
\iinner{\alpha, \beta}_{\varphi} &=\int_D \inner{\alpha,\beta}e^{\varphi} d\lambda,
\|\alpha \|^2_{\varphi}=\iinner{\alpha,\alpha}_{\varphi}.
\eal
$$
Let $L^2_{(0,q)} (D,\varphi)$ be the completion of $\wedge^{0,q}(D)$ with respect to the inner product $\iinner{\cdot,\cdot}_{\varphi}$,
then $L^2_{(0,q)}(D,\varphi)$ is a Hilbert space and
$\dbar:L^2_{(0,q)}(D,\varphi)\ra L^2_{(0,q+1)}(D,\varphi)$ is a  closed and densely defined operator.
Let $\dbar_\varphi^{*}$ be the Hilbert adjoint of $\dbar$.
For convenience,
we denote $L^2 (D, \varphi):=L^2_{(0,0)} (D, \varphi)$.

Set $H_1=[L^2 (D, \varphi)]^{\oplus p}$,
then for any $h,h^\prime \in H_1$,
we set
$$
	\langle h,h^\prime \rangle _{\varphi} = \sum^{p} _{i=1} \langle h_i, h^\prime_i \rangle _{\varphi},\quad
	|h |^2_{\varphi}=\langle h, h \rangle _{\varphi}
$$
for any $h=(h_1,\cdots,h_p), h^\prime=(h^\prime_1,\cdots,h^\prime_p) \in H_1.$

Suppose that $g=(g_1, g_2,\cdots, g_p) \in \calO(D)^{\oplus p}$ with $0<\inf_D|g|^2\le\sup_D|g|^2<+\infty$ and $\sup_D|\pa g|^2<+\infty$.
Set $H_2=L^2(D,\varphi)$ and we define an operator $T_{1}$ as follows:
$$
\begin{aligned}
	T_{1}:H_1 &\ra H_2,
	\\
	h=(h_1,h_2,\cdots,h_p) & \ra \sum_{i=1}^{p} g_i h_i.
\end{aligned}
$$
Then $T_1$ is a continuous linear operator since $g$ is bounded on $D$.

Set
$H_3=[L^2_{(0,1)} (D, \varphi)]^{\oplus p}$, and
we define an operator $T_2$ as follows:
$$
\begin{aligned}
	T_2:H_1
	   &\ra H_3,\\
	   h=(h_1,h_2,\cdots,h_p) &\ra
	    (\dbar h_1,\dbar h_2,\cdots,\dbar h_p).
\end{aligned}
$$
By definition,
$T_1$ sends $\Ker (T_2)$ to the closed subspace
 $L^2 (D, \varphi)\cap \Ker(\dbar)$.

We consider the following situation:
\begin{equation}\label{aaa}
  \begin{matrix}
	\Ker (T_2) & \subset & H_1  & \xrightarrow{T_1} & H_2 \\
	& & \downarrow_{T_2} & & \\
	& & H_3 & &
\end{matrix}.
\end{equation}

The division problem is to
find $h \in  \Ker (T_2)$
such that
$$
T_1 h=\sum_{i=1}^{p} h_i g_i=f.
$$
 for any given $f \in L^2 (D,\varphi)\cap\Ker(\dbar)$.

Now let us recall a basic functional lemma, which can be obtained by Banach's closed-field theorem and Riesz's representation theorem.

\begin{lem}[\cite{Hor65}]
	\label{lem:Hormander-lemma}Let $T_1:H_1\to H_2$ be
	a continuous linear operator between two Hilbert spaces $H_1,H_2$.
	Assume that $F$ is a closed subspace of $H_2$ containing $\text{Im}(T_{1})$, then $\text{Im}(T_1)=F$ if and only if
	there exists $\delta>0$ such that
		$\delta \|f \|_{H_2} \leq \|T^* _1 f \|_{H_1}$
	holds for all $f \in F $.
	\end{lem}

Applying Lemma \ref{lem:Hormander-lemma}, we can prove the following lemma which is owing to Skoda.

\begin{lem}[\cite{Skoda}]\label{lem:functional-lemma-Skoda}
	Let $H_1,H_2,H_3$ be Hilbert spaces,
    $T_1:H_1 \ra H_2$ be a continuous linear operator and
    $T_2:H_1 \ra H_3$ be a closed and densely defined
	operator with $T_{1}(\Ker (T_2)) \subset F$,
where $F$ is a closed subspace of $H_2$.
	Then $T_1(\Ker(T_2))=F$ if and only if
	there exists a constant $\delta >0$
        such that
		$$
        \| u \|_{H_2} \leq \delta \| T^*_{1} u+T^*_{2} v \|_{H_1}
		$$
		holds for all $ u \in  F$
		and $ v \in \Dom(T^*_2)\cap (\Ker(T_2^*))^\bot$.
\end{lem}

\begin{proof}
	Consider the linear continuous operator
	$$
    T:=T_{1} \big |_{\Ker(T_2)}: \Ker(T_2) \ra  F.
	$$
	
	It follows from Lemma \ref{lem:Hormander-lemma} that
	$\Img (T)=F$ is closed if and only if
	there exists a constant $\delta >0$
	such that
	$$
	\| u \|_{H_2} \leq \delta \| T^* u \|_{\Ker (T_2)}
	$$
	holds for any $u \in F$.
	Next we compute $\| T^* u \|_{\Ker (T_2)}$.

	Since for any $u \in \Ker(T_2), v \in F$,
	we have
	$$
	\begin{aligned}
         (Tu, v)_{F} &= (T_1u, v)_{H_2}  \\
                               &=(u, T^* _1 v)_{H_1} \\
                               &=(u, \text{pr}\circ T^*_1 v )_{H_1} \\
                                &=(u, \text{pr}\circ T^*_1 v )_{\Ker(T_2)},
	\end{aligned}
	$$
	where  $\text{pr}:H_1 \ra \Ker (T_2)$ be the projection map. Therefore, $T^*= \text{pr}\circ T^*_1$.

Hence
$$
\begin{aligned}
	\|T^* u \|_{\Ker (T_2)} &= \|\text{pr}\circ T^*_1 u \|_{\Ker (T_2)}  \\
	&= \inf_{\alpha \in (\Ker(T_2))^{\bot}}
	                           \|T^*_1 u+\alpha \|_{H_1}\\
					&= \inf_{v \in \Dom (T^*_2)} \|T^*_1 u+ T^*_2 v \|_{H_1} \quad (\text{since}~ (\Ker (T_2))^{\bot}=\overline{\text{Im}T^*_2})\\
&= \inf_{v \in \Dom (T^*_2)\cap(\Ker(T_2^*))^\bot} \|T^*_1 u+ T^*_2 v \|_{H_1}.
\end{aligned}
$$
We complete the proof.
\end{proof}

Therefore, in order to solve the division problem, it suffices to
 verify the inequality
\begin{equation}\label{eq:main-functional}
	\delta \| u\|^2_\varphi
\leq  \|T_1^*u+T_2^*v\|_{\varphi}^2
\end{equation}
for any $u \in \Ker(\dbar)\cap L^2(D,\varphi),
v \in  \Dom (T_2^*)\cap(\Ker(T_2^*))^\bot$ and
some constant $\delta >0$.

In the situation (\ref{aaa}),  we have
$$
T^* _{2} v=(\dbar^*_{\varphi} v_1,\cdots,\dbar^*_{\varphi}  v_p)
$$
for any $v\in \Dom (T^*_2)$. Since
$$
\iinner{T_1 h,u}_{\varphi}=\int_{D} \sum_{i=1}^{p}g_{i}h_{i}\overline{u}e^{-\varphi}d\lambda
	=\sum_{i=1}^{p}\int_{D}h_{i}\overline{(\bar{g_{i}} u )} e^{-\varphi}d\lambda
$$
for any $h\in H_1$, we get for any $u\in H_2$,
$$
\begin{aligned}
	T_1^* u =(\bar{g}_1 u,\bar{g}_2 u ,\cdots,\bar{g}_p u).
\end{aligned}
$$

By a direct computation, we have the following inequality:

\begin{lem}\label{lem:skoda-basic-inequality}
\begin{equation}\label{eq:Skoda-inequality}
	\bal
\| T_1^* u+ T_2^ * v\|_{\varphi}^2  \geq& \| T^*_2 v \|^2 _{\varphi} +
			(1-\frac{1}{b} )\int_D |g|^2 |u|^2 e^{-\varphi}d\lambda \\
			&-b\int_{D}|g|^{-2} \left|
			\sum^{p}_{i,k=1}
			\frac{\pa g_i }{\pa z_k} v_{ik}
		\right|^2   e^{-\varphi}  d\lambda.
\eal
\end{equation}
holds for any $b >1$, $u \in \Ker (\dbar)\cap L^2(D,\varphi),
v \in \Dom(T_2^*)\cap(\Ker(T_2^*))^\bot$,
where $v=(v_1,\cdots ,v_p), v_i=\sum^{n}_{j=1} v_{ij}d\bar z_j$.
\end{lem}

\begin{proof}

Firstly,
	$$
	\| T_1^* u+ T_2^ * v\|_{\varphi}^2=\| T_1^*u \|_{\varphi}^2
	+2\text{Re}\iinner{T_1^*u,T_2^*v}_{\varphi}+\| T_2^*v\|_{\varphi}^2,
	$$
where
$$
			\| \mathrm{T}_1^* u \|_{\varphi}^2
			=  \int_D |g|^2 |u|^2  e^{-\varphi}d\lambda
$$
and
$$
\begin{aligned}
			2\text{Re}\iinner{T_1^*u,T_2^*v }_{\varphi}
			=&2\text{Re} \iinner{\overline{g} u  ,T^*_2 v }_{\varphi}
			=2\text{Re}  \iinner{ T_2 (\overline{g} u ), v }_{\varphi} \\
			=&2\text{Re} \sum^{p}_{i=1}\iinner{ \dbar (\overline{g_i} u), v_i }_{\varphi}.
		\end{aligned}
$$
Since $u \in \Ker(\dbar)$,
we have
$$
\bal
\dbar(\overline{g_i} u ) &=\dbar (\overline{g_i}) \cdot u+\overline{g_i}\dbar u \\
&=\dbar (\overline{g_i}) \cdot u.
\eal
$$
Hence
$$
	\begin{aligned}
		2\text{Re}\iinner{T_1^*u,T_2^*v}_{\varphi}
		&=2\text{Re} \sum^{p}_{i=1}\iinner{ u \cdot \dbar (\overline{g_i}), v_i }_{\varphi}  \\
		&=2\text{Re} \int_{D} u \cdot
		\left(
			\overline{
			\sum^{p}_{i,k=1}
			\frac{\pa g_i }{\pa z_k} v_{ik}}
		\right)
		e^{-\varphi} d\lambda.
		\end{aligned}
$$
Then by the Cauchy-Schwarz inequality, for any $b >1$ we obtain
	$$
	\begin{aligned}
		 \left|2\text{Re}\left(  u
			\overline{
			\sum^{p}_{i,k=1}
			\frac{\pa g_i }{\pa z_k} v_{ik}}
			\right)
		\right|  &
		   \leq \frac{1}{b} |g|^2 |u|^2
		   + b\frac{1}{|g|^2} \left|
			\sum^{p}_{i,k=1}
			\frac{\pa g_i }{\pa z_k} v_{ik}
		\right|^2.
	\end{aligned}
	$$
	Thus
	$$
		\begin{aligned}
			2\text{Re}\iinner{T_1^*u,T_2^*v}_\varphi \geq &
			-\frac{1}{b} \int_{D}|g|^2|u|^2e^{-\varphi}d\lambda
			-b\int_{D}\frac{1}{|g|^2} \left|
			\sum^{p}_{i,k=1}
			\frac{\pa g_i }{\pa z_k} v_{ik}
		\right|^2 e^{-\varphi}d\lambda.
			\end{aligned}
	$$
   In summary,
	$$
	\bal
			\| T_1^* u+ T_2^ * v\|_{\varphi}^2  \geq & \| T^*_2 v \|^2 _{\varphi} +
			(1-\frac{1}{b} )\int_D |g|^2 |u|^2 e^{-\varphi}d\lambda \\
			&-b\int_{D}|g|^{-2} \left|
			\sum^{p}_{i,k=1}
			\frac{\pa g_i }{\pa z_k} v_{ik}
		\right|^2   e^{-\varphi}  d\lambda.
		\eal
	$$

\end{proof}

\begin{lem}\label{lem:weak-basic-inequality}
Let $D\subset\mc^n$ be a domain and $\psi$ be a real-valued continuous function on $D$.
	Assume that $(D,\psi)$ is $L^2$-optimal,
	then
	\begin{equation}\label{eq:basic-inequality-1}
		\iinner{ B_{\phi} \alpha,\alpha }_{\phi+\psi}
			\leq	
		\|\dbar^{*}_{\phi+\psi} \alpha \|^2_{\phi+\psi}.
	\end{equation}
	holds for  any $(0,1)$-form
$\alpha \in \Dom(\dbar^*_{\phi+\psi}) \cap \Ker(\dbar)$ and any smooth function $\phi$ on $\mc^n$ of the form
$\phi(z)=a\|z\|^2,\ a>0$.

\end{lem}

\begin{proof}
	
For any $\alpha, f \in  \Dom(\dbar^*_{\phi+\psi})
	\cap \Ker (\dbar)$,
the assumption that  $(D,\psi)$ is $L^2$-optimal implies that  there exists $u \in L^2(D,\phi+\psi)$ such that
\begin{equation}\label{eq:basic-inequality-2}
\begin{aligned}
	|\iinner{  \alpha, f}_{\phi+\psi} |^{2}
	&=|\iinner{  \alpha,\dbar u }_{\phi+\psi}|^{2}
	=|\iinner{ \dbar^{*}_{\phi+\psi} \alpha,u }_{\phi+\psi}|^{2}\\
	& \leq \|u  \|_{\phi+\psi}^{2} \cdot \|\dbar^{*}_{\phi+\psi} \alpha \|_{\phi+\psi}^{2}  \\
	& \leq  (
		\int_D \langle B_{\phi}^{-1}f,f\rangle e^{-\phi-\psi} d\lambda
	)\cdot
	\|\dbar^{*}_{\phi+\psi} \alpha \|^2_{\phi+\psi}.
\end{aligned}
\end{equation}

Now for any strictly plurisubharmonic function $\phi$ on $\mc^n$ of the form
$\phi(z)=a\|z\|^2,\ a>0,$
we have $B_{\phi} \alpha =a\alpha\in \Ker(\dbar)$ since
$\alpha \in \Ker (\dbar)$.
Then by substituting $f$ in \eqref{eq:basic-inequality-2}
with $B_{\phi} \alpha$,
we obtain that
$$
\iinner{B_{\phi} \alpha,\alpha}_{\phi+\psi}
\leq	
\|\dbar^{*}_{\phi+\psi} \alpha \|^2_{\phi+\psi}.
$$
\end{proof}

\begin{rem}
When $D$ is complete K\"ahler, \eqref{eq:basic-inequality-1} is a direct corollary of the Bochner-Kodaira-Nakano identity.
\end{rem}

\begin{thm}[Theorem \ref{thm:1}]\label{thm:skoda-division}
Let $D\subset\mc^n$ be a  domain and $\psi$ be a real-valued continuous function on $D$.
	Assume that $(D,\psi)$ is $L^2$-optimal. Suppose that $g\in\calO(D)^{\oplus p}$ with $0<\inf_D|g|^2\le\sup_D|g|^2<+\infty$ and $\sup_D|\pa g|^2<+\infty$. Then there is a smooth function $\phi$ on $\mc^n$ of the form
$\phi(z)=a\|z\|^2,\ a>>1$  and a constant $\delta>0$ such that
\begin{equation}\label{eq:Skoda-division}
	\bal
\| T_1^* u+ T_2^ * v\|_{\phi+\psi}^2  \geq&\delta \|u\|_{\phi+\psi}^2
\eal
\end{equation}
 for all $u \in \Ker (\dbar) \cap \Dom(T_1^*)$ and
$v \in \Dom(T_2^*)\cap(\Ker(T_2^*))^\bot$. In particular,  for any holomorphic function $f \in L^2 (D,\phi+\psi)$, there are $p$ holomorphic functions $h_i\in L^2 (D,\phi+\psi), i=1,\cdots,p,$
such that
$$
\sum_{i=1}^{p} h_i g_i=f.
$$
\end{thm}
\begin{proof}
Since $(\Ker(T_2^*))^\bot=\overline{\Img(T_2)}$, then for any $v=(v_1,\cdots ,v_p) \in \Dom(T_2^*)\cap(\Ker(T_2^*))^\bot$, $v_i=\sum^{n}_{j=1} v_{ij}d\bar z_j$ is $\dbar$-closed $(0,1)$-form.  Then  for  any function $\phi$ on $\mc^n$ of the form
$\phi(z)=a\|z\|^2,\ a>0,$ and any $v=(v_1,\cdots ,v_p) \in \Dom(T_2^*)\cap(\Ker(T_2^*))^\bot$, it follows from Lemma \ref{lem:weak-basic-inequality} that
\begin{align*}
 \| T^*_2 v \|^2 _{\phi+\psi}= & \sum_{i=1}^{p}\| \dbar^*_{\phi+\psi} v_i \|^2 _{\phi+\psi} \\
    \ge & \sum_{i=1}^{p} \iinner{B_{\phi} v_i,v_i}_{\phi+\psi}\\
     = & \sum_{i=1}^{p}a\|v_i\|^2_{\phi+\psi}=a\|v\|^2_{\phi+\psi}.
 \end{align*}
 Then by  Lemma \ref{lem:skoda-basic-inequality}, we obtain that
 $$
	\bal
			\| T_1^* u+ T_2^ * v\|_{\phi+\psi}^2  \geq & \| T^*_2 v \|^2 _{\phi+\psi} +
			(1-\frac{1}{b} )\int_D |g|^2 |u|^2 e^{-\phi-\psi}d\lambda \\
			&-b\int_{D}|g|^{-2} \left|
			\sum^{p}_{i,k=1}
			\frac{\pa g_i }{\pa z_k} v_{ik}
		\right|^2   e^{-\phi-\psi}  d\lambda\\
 \geq & a\|v\|^2_{\phi+\psi} +
			(1-\frac{1}{b} )\inf_D|g|^2\|u\|^2_{\phi+\psi} -bp^2\sup_D\frac{|\partial g|^2}{|g|^{2}}\|v\|^2_{\phi+\psi}\\
\geq & (1-\frac{1}{b} )\inf_D|g|^2\|u\|^2_{\phi+\psi},
		\eal
	$$
where we can take $a=bp^2\sup_D\frac{|\pa g|^2}{|g|^{2}}<+\infty$ and $\delta= (1-\frac{1}{b} )\inf_D|g|^2>0$.
\end{proof}

\section{A new characterization of $L^2$-domain of holomorphy}
\label{sec:proof of main}
In this section, we will give a proof of Theorem \ref{thm:main} by using Theorem \ref{thm:skoda-division} and generalize these results to $L^2$-optimal domains in Stein manifolds.

\begin{thm}[Theorem \ref{thm:main}]
	Let $D$ be a bounded domain or a tube domain associated to a bounded domain in $\mr^n$ in $\mc^n$ with a null thin complement. Assume that $D$ is $L^2$-optimal,
then $D$ is an $L^2$-domain of holomorphy.
\end{thm}

\begin{proof}Suppose that  $D$ is not an $L^2$-domain of holomorphy. Then by the definition,  there is a connected  open set $U_2$ with $U_2 \cap D\neq\emptyset$ and $U_2 \not\subset D$ and a open subset $U_1\subset U_2\cap D$ such that for every holomorphic function $f\in L^2(D)$, there is an $F\in \calO(U_2)$ satisfying $F=f$ on $U_1$.

Since  $\mathring{\overline {D}}=D$, we obtain that $U_2 \not\subset \overline {D}$. Thus for some $z^0\in U_2 \backslash \overline {D}$, in Theorem \ref{thm:skoda-division}, we can take $g(z)=z-z^0$ when $D$ is bounded (or $g(z)=e^{z-z^0}-1$ when $D$ is a tube domain associated to a bounded domain) . Then $0<\inf_D|g|^2\le\sup_D|g|^2<+\infty$ and $\sup_D|\pa g|^2<+\infty$. Notice that $f\equiv1\in L^2(D,a|z|^2)$ for $a$ large enough. Since $D$ is $L^2$-optimal, then by Theorem \ref{thm:skoda-division}, there are $n$ holomorphic functions $h_i\in L^2 (D,a|z|^2), i=1,\cdots,n,$
such that on $D$,
$$
\sum_{i=1}^{n} h_i(z)(z_i-z^0_i) \equiv1.
$$
 Then there exist $H_i\in\calO(U_2), i=1,\cdots,n,$ such that $$H_i\big|_{U_1}=h_i\big|_{U_1}.$$
Then on $U_1$ we have
$$
\sum_{i=1}^{n} H_i(z)(z_i-z^0_i) \equiv1.
$$
Hence by the uniqueness theorem for holomorphic functions, we obtain $\sum_{i=1}^{n} H_i(z)(z_i-z^0_i) \equiv1$
on $U_2$.
In particular, $\sum_{i=1}^{n} H_i(z)(z_i-z^0_i) =1$ at $z^0\in U_2$, which is a contradiction.

\end{proof}

Now, let us recall the notations for the $\dbar$-equation on complex manifolds and generalize the above results on Stein manifolds. 

Let $(X,\omega)$ be a K\"ahler manifold of dimension $n$ and $L$ be a Hermitian holomorphic line bundle endowed with a continuous metric $e^{-\varphi}$ over $X$.
Let $|\cdot|_{\omega,\varphi}$ and $\inner{\cdot,\cdot}_{\omega,\varphi}$ denote the norm and the inner product on $\wedge^{p,q}T^*M\otimes L$ that induced by $\omega$ and $e^{-\varphi}$.
$dV_\omega:=\frac{\omega^n}{n!}$ denotes the volume form induced by $\omega$.
Let $\wedge_c^{p,q}(X,L)$ be the space of all compactly supported smooth $L$-valued $(p,q)$-forms on $X$,
and $L_{p,q}^2(X,L;\loc)$ be the space of all  $L$-valued $(p,q)$-forms on $X$ whose coefficients are $L_\loc^2$ functions. Define
$$
	L_{p,q}^2(X,L;\omega,\varphi) := \left\{f\in L_{p,q}^2(X,L;\loc): \int_X |f|_{\omega,\varphi}^2dV_\omega <+\infty \right\},
$$
and
$$
\iinner{\cdot, \cdot }_{\omega, \varphi}= \int_{X}\inner{\cdot,\cdot}_{\omega,\varphi} dV_\omega
$$

Then $L_{p,q}^2(X,L;\omega,\varphi) $ is a Hilbert space containing $\wedge_c^{p,q}(X,L)$ as a dense subspace, and $\dbar: L_{p,q}^2(X,L;\omega,\varphi)  \rightarrow L_{p,q+1}^2(X,L;\omega,\varphi) $ is a  closed and densely defined operator.

Let $D$ be a relatively compact domain in $X$.
Set $H_1=[L_{(n,0)}^2(D,L;\omega,\varphi)]^{\oplus p}$,
then for any $h,h'\in H_1$,
we set
$$
 \iinner{	h, h' }_{\omega,\varphi} = \sum^{p} _{i=1} \iinner{ h_i, h'_i }_{\omega,\varphi},\quad
	\|h \|^2_{\omega,\varphi}=\iinner{ h, h } _{\omega,\varphi}
$$
for any $h=(h_1,\cdots,h_p), h'=(h'_1,\cdots,h'_p) \in H_1.$

Take $g=(g_1, g_2,\cdots, g_p) \in \calO(D)^{\oplus p}$ with $0<\inf_D|g|^2\le\sup_D|g|^2<+\infty$ and $\sup_D|\pa g|^2<+\infty$.
Set $H_2=L^2_{(n,0)}(D,L;\omega,\varphi)$.
We define an operator $T_{1}$ as follows:
$$
\begin{aligned}
	T_{1}:H_1 &\ra H_2,
	\\
	h=(h_1,h_2,\cdots,h_p) & \ra \sum_{i=1}^{p} g_i h_i.
\end{aligned}
$$
Then $T_1$ is a continuous linear operator since $g$ is bounded on $D$.

Set
$H_3=[L^2_{(n,1)} (D,L;\omega, \varphi)]^{\oplus p}$, and
we define an operator $T_2$ as follows:
$$
\begin{aligned}
	T_2:H_1
	   &\ra H_3,\\
	   h=(h_1,h_2,\cdots,h_p) &\ra
	    (\dbar h_1,\dbar h_2,\cdots,\dbar h_p).
\end{aligned}
$$
By definition,
$T_1$ sends $\Ker (T_2)$ to the closed subspace
 $L^2_{(n,0)} (D;\omega, \varphi)\cap \Ker(\dbar)$.
And since  for any $h\in H_1$,
$$
\iinner{T_1 h,u}_{\omega,\varphi}=\int_{D} \sum_{i=1}^{p}g_{i}h_{i}\wedge\overline{u}e^{-\varphi}
	=\sum_{i=1}^{p}\int_{D}h_{i}\wedge\overline{(\bar{g_{i}} u )} e^{-\varphi}
=\sum_{i=1}^{p}\iinner{h_i,\bar{g_{i}} u}_{\omega,\varphi},
$$
we get for any $u\in H_2$
$$
\begin{aligned}
	T_1^* u =(\bar{g}_1 u,\bar{g}_2 u ,\cdots,\bar{g}_p u).
\end{aligned}
$$

Similar to Lemma \ref{lem:skoda-basic-inequality},
we have the following inequality.
\begin{lem}\label{lem:skoda-basic-inequality'} With notations as above, we have
\begin{equation}\label{eq:Skoda-inequality-2}
	\bal
	\| T_1^* u+ T_2^ * v\|_{\omega,\varphi}^2  \geq & \| T^*_2 v \|^2 _{\omega,\varphi} +(1-\frac{1}{b} )\inf_D|g|^2\|u\|^2_{\omega,\varphi} -bp^4\sup_D\frac{|\pa g|^2}{|g|^{2}}\|v\|^2_{\omega,\varphi}
	\eal
\end{equation}
holds for any $b >1$, $u \in \Ker (\dbar) \cap \Dom(T_1^*),
v \in \Dom(T_2^*)\cap(\Ker(T_2^*))^\bot$.
\end{lem}
Now let us introduce the definition of \emph{$L^2$-optimal} domains in a Stein manifold.
\begin{defn}
   Let $D$ be a domain in a Stein manifold $X$ and $L$ be a holomorphic line bundle endowed with a continuous metric $e^{-\psi}$ over $D$.
  We say that  a pair $(D,L,\psi)$ is \emph{$L^2$-optimal}
   if for any smooth strictly plurisubharmonic function $\phi$ and
   any K\"{a}hler metric $\omega$ on $X$,
   the equation $\dbar u=f$ can be solved on $D$ for any $\dbar$-closed $L$-valued $(n,1)$-form
   $f \in L^2_{(n,1)}(D,\phi)$
   with the estimate:
   $$
   \int_D|u|^2_{\omega,\psi} e^{-\phi} dV_\omega
   \leq
   \int_D \inner{B_\phi^{-1}f,f}_{\omega,\psi} e^{-\phi} dV_{\omega},
   $$
   provided the right hand side is finite, where $B_\phi:=[i\ddbar \phi, \Lambda_{\omega}]$.

   In particular, we call $D$ is  $\emph{$L^2$-optimal}$ if the pair $(D,D\times\mc,0)$ is \emph{$L^2$-optimal}.
\end{defn}

It is well-known that every Stein manifold $X$ admits a smooth strictly plurisubharmonic exhaustive function $\phi$ such that
 $i\ddbar \phi$ defines a (complete) K\"ahler metric on $X$ (\cite{Docquier-Grauert,Grauert}).
Noticing that  $[i\ddbar \phi, \Lambda_{i\ddbar \phi}]u=(p+q-n)u$ for any $u\in\wedge^{p,q}T_X^*$, we can obtain a weak version of the Bochner-type inequality similar to Lemma \ref{lem:weak-basic-inequality} as follows.

\begin{lem}\label{lem:weak-basic-inequality'}
Let $D$ be a domain in a Stein manifold $X$ and $L$ be a Hermitian holomorphic line bundle endowed with a continuous metric $e^{-\psi}$ over $D$. Let $\phi$ be a smooth strictly plurisubharmonic exhaustive function on $X$. Take $\omega=i\ddbar\phi$.
	Assume that $(D,\psi)$ is $L^2$-optimal,
	then
	\begin{equation}\label{eq:basic-inequality}
		\iinner{ B_{a\phi} \alpha,\alpha }_{\omega,a\phi+\psi}
			\leq	
		\|\dbar^{*}_{a\phi+\psi} \alpha \|^2_{a\phi+\psi}.
	\end{equation}
	holds for $a>0$ and any $(n,1)$-form
$\alpha \in \Dom(\dbar^*_{a\phi+\psi}) \cap \Ker(\dbar)$.

\end{lem}

\begin{proof}
	
For any $\alpha, f \in  \Dom(\dbar^*_{a\phi+\psi})
	\cap \Ker (\dbar)$,
the assumption that  $(D,\psi)$ is $L^2$-optimal implies that  there exists $u \in L^2_{n,0}(D,L;\omega,a\phi+\psi)$ such that
\begin{equation}\label{eq:basic-inequality-1'}
\begin{aligned}
	|\iinner{  \alpha, f}_{\omega,a\phi+\psi} |^{2}
	&=|\iinner{  \alpha,\dbar u }_{\omega,a\phi+\psi}|^{2}
	=|\iinner{ \dbar^{*}_{a\phi+\psi} \alpha,u }_{\omega,a\phi+\psi}|^{2}\\
	& \leq \|u  \|_{\omega,a\phi+\psi}^{2} \cdot \|\dbar^{*}_{\omega,a\phi+\psi} \alpha \|_{\omega,a\phi+\psi}^{2}  \\
	& \leq  (
		\int_D \langle B_{a\phi}^{-1}f,f\rangle_{\omega,\psi} e^{-a\phi} dV_\omega
	)\cdot
	\|\dbar^{*}_{a\phi+\psi} \alpha \|^2_{\omega,a\phi+\psi}.
\end{aligned}
\end{equation}

Notice that $B_{a\phi}\alpha=[i\ddbar (a\phi),\Lambda_\omega]\alpha=a[i\ddbar \phi,\Lambda_{i\ddbar \phi}]=a\alpha\in \Ker(\dbar)$, 
then by substituting $f$ in \eqref{eq:basic-inequality-1'}
with $B_{a\phi} \alpha$,
we obtain that
$$
\iinner{B_{a\phi} \alpha,\alpha}_{\omega,a\phi+\psi}
\leq	
\|\dbar^{*}_{a\phi+\psi} \alpha \|^2_{\omega,a\phi+\psi}.
$$
\end{proof}

Therefore, it follows from Lemma \ref{lem:skoda-basic-inequality'} , Lemma \ref{lem:weak-basic-inequality'} and Lemma \ref{lem:functional-lemma-Skoda} that
\begin{thm}\label{thm:skoda-division'}
Let $D$ be a domain in a Stein manifold $X$ and $L$ be a holomorphic line bundle endowed with a continuous metric $e^{-\psi}$ over $D$. Let $\phi$ be a smooth strictly plurisubharmonic exhaustive function on $X$. Take $\omega=i\ddbar\phi$.
	Assume that $(D,\psi)$ is $L^2$-optimal. Suppose that $g\in\calO(D)^{\oplus p}$ with $0<\inf_D|g|^2\le\sup_D|g|^2<+\infty$ and $\sup_D|\pa g|^2<+\infty$. Then there are two constants $a, \delta>0$ such that
\begin{equation}\label{eq:Skoda-division'}
	\bal
\| T_1^* u+ T_2^ * v\|_{\omega,a\phi+\psi}^2  \geq&\delta \|u\|_{\omega,a\phi+\psi}^2
\eal
\end{equation}
 for all $u \in \Ker (\dbar) \cap \Dom(T_1^*)$ and
$v \in \Dom(T_2^*)\cap(\Ker(T_2^*))^\bot$. In particular,  for any holomorphic section $f \in L^2_{(n,0)} (D,L;\omega,a\phi+\psi)$, there are $p$ holomorphic sections $h_i\in L^2_{(n,0)} (D,L;\omega,a\phi+\psi), i=1,\cdots,p,$
such that
$$
\sum_{i=1}^{p} h_i g_i=f.
$$
\end{thm}

\begin{defn}[\cite{PW}]
	Let $D$ be a domain in a complex manifold.
	We say that $D$ is a $L^2$-domain of holomorphy if for any pair of open sets $U_1, U_2$ with $\emptyset \neq U_1 \subset D \cap U_2 \neq U_2$,
	$U_2$ connected,
	there exists a holomorphic $(n,0)$-form $f \in L_{n,0}^2 ( D)$ such that for any holomorphic $(n,0)$-form $F$ on $U_2$,
	we have
	$$
	f \big|_{U_1} \neq F\big|_{U_1}.
	$$
\end{defn}

Let $D$ be a $L^2$-domain of holomorphy in a complex manifold $X$. Then
for any open coordinate set $U$,  $D\cap U$ can be seen as a domain of holomorphy in $\mc^n$. Hence $D$ is locally Stein. 
In particular,  any  $L^2$-domain of holomorphy in a Stein manifold is Stein and hence $L^2$-optimal. Conversely, we have

\begin{thm}\label{thm stein}
	Let $D$ be a relative compact domain in a  Stein manifold $X$ with a null thin complement.
Then $D$ is $L^2$-optimal if and only if $D$ is an $L^2$-domain of holomorphy.
\end{thm}

\begin{proof}Suppose that  $D$ is not an $L^2$-domain of holomorphy. Then by the definition, there is a connected  open  set $U_2$ with $U_2 \cap D\neq\emptyset$ and $U_2 \not\subset D$ and a open subset $U_1\subset U_2\cap D$ such that for every holomorphic $(n,0)$-form $f\in L^2(D)$, there is a   holomorphic $(n,0)$-form $F$ on $U_2$ satisfying $F=f$ on $U_1$.

Let $\phi$ be a smooth  strictly plurisubharmonic exhaustive function on $X$. Take $\omega=i\ddbar\phi$. Since  $\mathring{\overline {D}}=D$, we obtain that $U_2 \not\subset \overline {D}$. Then there is $p\in U_2 \backslash \overline {D}$. Consider the proper holomorphic embedding
	$$
	\tau : X \ra \mc^{2n+1}
	$$
and take $g=\tau^* (z-\tau(p)) \in \calO(X)^{\oplus 2n+1}$. Then $g(p)=0$, $0<\inf_D|g|^2\le\sup_D|g|^2<+\infty$ and $\sup_D|\pa g|^2<+\infty$. In Theorem \ref{thm:skoda-division'}, take $(L,e^{-\psi})=(X\times \mc,1)$. By Cartan's theorem $A$, there is a holomorphic section $f\in H^0(X,K_X)$ such that $f(p)\neq0$. Since $D$ is bounded, we have
\begin{align*}
  \int_{D}^{}|f|^2_\omega e^{-a\phi}dV_\omega<+\infty.
\end{align*}
Since $D$ is $L^2$-optimal, then by Theorem \ref{thm:skoda-division'}, there are $2n+1$ holomorphic sections $h_i\in L^2_{(n,0)} (D;\omega), i=1,\cdots,2n+1,$
such that on $D$,
$$
\sum_{i=1}^{2n+1} h_ig_i=f.
$$
Then there exist  holomorphic $(n,0)$-form $H_i $ on $U_2$, $i=1\cdots,2n+1,$ such that $$H_i\big|_{U_1}= h_i\big|_{U_1}.$$
Then on $U_1$ we have
$$
\sum_{i=1}^{2n+1} H_i(z)g_i(z)=f(z).
$$
Hence by the uniqueness theorem for holomorphic functions, we obtain $\sum_{i=1}^{2n+1} H_i(z)g_i(z)= f(z) $
on $U_2$.

In particular, $0=\sum_{i=1}^{n} H_i(p)g(p) =f(p)\neq0$, which is a contradiction.
\end{proof}

\section{Further remarks}\label{sec:Negligible-set}

In this section, we will show that $L^2$-optimal domains are more flexible than complete K\"ahler domains.

Firstly, let us recall the notable H\"ormander $L^2$ existence theorem. 
\begin{thm}[$L^2$ existence theorem, \cite{Dem-82, Hor65}]\label{thml2existence}
  Let $X$ be a complete K\"ahler manifold, $\omega$ a K\"ahler metric (possibly non complete)  and $L$ be a Hermitian holomorphic line bundle endowed with a smooth semi-positive metric $e^{-\psi}$ over $X$. Assume that $\phi$ is a smooth strictly plurisubharmonic function on $X$. Then for any $f\in L^2_{(n,1)}(X,L;\omega,\psi+\phi)\cap\Ker\dbar$ satisfying with
  \begin{align*}
              \int_D \inner{B_\phi^{-1}f,f}_{\omega,\psi+\phi}  dV_{\omega}   < +\infty,
  \end{align*} there is a
  $u\in L^2_{(n,0)}(X,L;\omega,\psi+\phi)$ such that $\dbar u=f$
  and
   $$
   \int_D|u|^2_{\omega,\psi+\phi}  dV_\omega
   \leq
   \int_D \inner{B_\phi^{-1}f,f}_{\omega,\psi+\phi}  dV_{\omega},
   $$
    where $B_\phi^{-1}:=[i\ddbar \phi, \Lambda_{\omega}]$.
\end{thm}

By a standard procedure of functional analysis, one can show that
\begin{prop}\label{prop closedness}
  Let $D$ be a domain in $\mc^n$ and
$\psi$  a upper semi-continuous function defined on $ D$.  Let $\{D_j\}$ be a sequence of open subsets of $D$ with $D_j\subset D_{j+1} $ and $\bigcup_jD_j=D$, and $\{\psi_j\}$ be a sequence of upper semi-continuous functions defined on $D_j$  with $\psi_j\ge\psi_{j+1}$ and $\lim_{j\to+\infty}\psi_j=\psi$.  Assume that all $(D_j,\psi_j)$ are $L^2$-optimal, then $(D,\psi)$ is also $L^2$-optimal.
\end{prop}
\begin{proof}
 Let $\phi$ be a smooth strictly plurisubharmonic function on $\mc^n$ and $f \in L^2_{(0,1)}(D,\phi+\psi)\cap\Ker(\dbar)$ satisfies
$$
\int_D|u|^2 e^{-\phi-\psi} d\lambda
\leq
\int_D \langle B_\phi^{-1} f, f\rangle e^{-\phi-\psi} d\lambda.
$$
Since $\psi_j\ge\psi$ and $(D_j,\psi_j)$ is $L^2$-optimal, then there exists $u_j\in L^2(D,\phi+\psi_j)$ such that $\dbar u_j=f$ on $D_j$ with the estimate
$$
\int_{D_j}|u_j|^2 e^{-\phi-\psi_j} d\lambda
\leq
\int_{D_j} \langle B_\phi^{-1} f, f\rangle e^{-\phi-\psi_j} d\lambda\leq
\int_D \langle B_\phi^{-1} f, f\rangle e^{-\phi-\psi} d\lambda.
$$
Let $K_l$ be an exhausting sequence of relatively compact subsets of $D$, which means that $K_l\Subset K_{l+1}$ and $\cup_l K_l=D$. Then for each $l$, $K_l\subset D_j$ for $j$ large enough. Since $\psi_1$ is upper semi-continuous and $\psi_j$ decreasingly converging to $\psi$, then for each $l$, $\{u_j\}$ is a bounded subset in $L^2(K_l)$.
Hence by the Banach-Alaoglu-Bourbaki theorem and the diagonal argument, we can take a sequence $l_j\to+\infty$ such that $u_{l_j}$ is weakly $L^2$ convergent on $K_l$ for every $l$ to a limit $u$ on $D$. Then it follows from the weakly closedness of $\dbar$ that $\dbar u=f$ on $D$. And by Fatou's lemma, for each $l$,
\begin{equation*}
    \int_{K_l}^{} |u|^2 e^{-\phi-\psi} d\lambda
    \le \liminf_{j\to +\infty}\int_{K_l}^{} |u_{l_j}|^2 e^{-\phi-\psi_j} d\lambda
    \int_D \langle B_\phi^{-1} f, f\rangle e^{-\phi-\psi} d\lambda.
    \end{equation*}
Thus, $$
\int_D|u|^2 e^{-\phi-\psi} d\lambda
\leq
\int_D \langle B_\phi^{-1} f, f\rangle e^{-\phi-\psi} d\lambda.
$$
\end{proof}

\begin{cor}\label{cor:aaa}
Let $D$ be a domain in a Stein manifold $X$ with a null thin complement. Assume that there is a sequence of open subsets $D_j$ of $D$ satisfying that $D_j\subset D_{j+1} $, $\bigcup_jD_j=D$ and every $D_j$ admits a complete K\"ahler metric. Then $D$ is Stein.
\end{cor}

\begin{proof}
  Let $X_l$ be an exhausting sequence of relatively compact Stein subsets with smooth boundarys of $X$. 
Then $D\cap X_l$ is  a relatively compact subset in a Stein manifold $X$ with a null thin complement. Since $D_j$ is complete K\"ahler, then $D_j\cap X_l$ is also complete K\"ahler and hence $L^2$-optimal by Theorem \ref{thml2existence}. Then by Proposition \ref{prop closedness}, $D\cap X_l$ is also $L^2$-optimal and hence an $L^2$-domain of holomorphy since $D\cap X_l$ has a null thin complement. In particular, $D\cap X_l$ is Stein. 
Therefore, $D$ is Stein.
  
\end{proof}

\begin{rem}
  In \cite{Yasuoka}, Yasuoka proved that a domain $D$ in a Stein manifold $X$ is Stein if $D$ can be exhausted by complete Kähler domains.
\end{rem}

\begin{lem}[\cite{CWW}]\label{lem L2negligible}
  Let $\Omega$ be an open subset of $\mc^n$ and $E$  a closed pluripolar subset of $\Omega$.
Assume that $v$ is  a $(p, q-1)$-form with $L^2_\loc$
coefficients and $w$ a $(p, q)$-form with $L^1_\loc$ coefficients
such that $\dbar v=w$ on $\Omega\backslash E$ (in the sense of distribution theory). Then $\dbar v=w$ on $\Omega$.
\end{lem}

Obviously, it follows from Theorem \ref{lem L2negligible} and the definition of $L^2$-optimal domains that

\begin{prop}\label{prop l2 neg}
Let $D_1\subset D_2$ be two domains in $\mc^n$ and
$\psi$ an upper semi-continuous  function defined on $ D_2$.
 Assume that there is a closed pluripolar subset $E \subset D_2$ such that $D_2=D_1\cup E$, then $(D_1,\psi)$ is $L^2$-optimal  if and only if $(D_2,\psi)$ is  $L^2$-optimal.
\end{prop}

\begin{proof}
$(\Rightarrow)$:
  Let $\phi$ be a smooth strictly plurisubharmonic function on $\mc^n$ and $f \in L^2_{(0,1)}(D_2,\phi+\psi)\cap\Ker(\dbar)$ satisfies
$$
\int_{D_2} \langle B_\phi^{-1} f, f\rangle e^{-\phi-\psi} d\lambda<+\infty.
$$
Since $D_1$ is $L^2$-optimal, then there exists $u\in L^2(D_1,\phi+\psi)$ such that $\dbar u=f$ on $D_1$ with the estimate
$$
\int_{D_1}|u|^2 e^{-\phi-\psi} d\lambda
\leq
\int_{D_1} \langle B_\phi^{-1} f, f\rangle e^{-\phi-\psi} d\lambda.
$$
Since $\psi$ an upper semi-continuous  function defined on $ D_2$, we obtain that $u$ and $f$ are both in $L^2_\loc$. Then by Lemma \ref{lem L2negligible}, $\dbar u=f$ can extend across the  closed pluripolar subset $E\subset D_2$ with the same $L^2$-estimate. Therefore, $(D_2,\psi)$ is $L^2$-optimal.

$(\Leftarrow)$: Let $\phi$ be a smooth strictly plurisubharmonic function on $\mc^n$ and $f \in L^2_{(0,1)}(D_1,\phi+\psi)\cap\Ker(\dbar)$ satisfies
$$
\int_{D_1} \langle B_\phi^{-1} f, f\rangle e^{-\phi-\psi} d\lambda<+\infty.
$$
Since $\psi$ an upper semi-continuous  function defined on $ D_2$, we obtain that $f$ is in $L^2_\loc$. Then by Lemma \ref{lem L2negligible}, $\dbar f=0$ can extend across the  closed pluripolar subset $E\subset D_2$ with the same $L^2$-estimate.

Since $D_2$ is $L^2$-optimal, then there exists $u\in L^2(D_2,\phi+\psi)$ such that $\dbar u=f$ on $D_2$ with the estimate
$$
\int_{D_2}|u|^2 e^{-\phi-\psi} d\lambda
\leq
\int_{D_2} \langle B_\phi^{-1} f, f\rangle e^{-\phi-\psi} d\lambda.
$$
Since $E$ is of zero measure and $D_1\subset D_2$, then $\dbar u=f$ on $D_1$ satisfying the same $L^2$-estimate. Therefore, $(D_1,\psi)$ is $L^2$-optimal.

\end{proof}

Then by Theorem \ref{thm stein} and Proposition \ref{prop l2 neg}, we can obtain that

\begin{cor}\label{cor L2neg}
	Let $D$ be a relatively compact $L^2$-optimal domain in a Stein manifold $X$.
	Assume that $\mathring{\overline{D}}\backslash D$ is pluripolar, then $\mathring{\overline{D}}$ is an $L^2$-domain of holomorphy.	
\end{cor}

Recall that there is a result on the boundary of complete K\"ahler domains obtained by 
Diederich-Forn\ae ss as follows.

\begin{thm}[\cite{D-F}]\label{thm df}
	Let $A$ be a closed real analytic subset of a pseudoconvex domain $D \subset \mc^n$. Suppose that $\textup{codim}\;A\ge3$, then $A$ is complex analytic if and only if $D \backslash A $ admits a complete K\"ahler metric.
\end{thm}

Then we can give an example that an $L^2$-optimal domain is not  complete K\"ahler by Theorem \ref{thm df} and Proposition \ref{prop l2 neg}.
\begin{exm}\label{exa}
	Let $A$ be a closed subset of $\mc$.
	Assume that $A$ is real analytic but not complex analytic and
  $$
   A \subset \Omega \subset \mc^n,
   $$
where $\Omega$ is a pseudoconvex domain in $\mc^n$ and $n \geq 4$.
Since $A \subset \mc$, then $A$ is closed pluripolar and $\text{codim} A \geq 3$.
Therefore, it follows from Theorem \ref{thm df} and Proposition \ref{prop l2 neg} that $\Omega \backslash A$ is $L^2$-optimal but not complete K\"ahler.
\end{exm}


\begin{thebibliography}{10}

	\bibitem{Andreotti-Vesentini}
	A.~Andreotti, E.~Vesentini,
	\newblock Carleman estimates for the Laplace-Beltrami equation on complex manifolds.
	\newblock {\em Inst. Hautes \'Etudes Sci. Publ. Math.}, 25:81--130 (1965)

    \bibitem{BP08}
    B.~Berndtsson, M.~P\u{a}un,
    \newblock Bergman kernels and the pseudoeffectivity of relative canonical
    bundles.
    \newblock {\em Duke Math. J.}, 145(2):341--378 (2008)


	\bibitem{Bremermann}
	H.~Bremermann,
	\newblock \"{U}ber die \"{A}quivalenz der pseudokonvexen {G}ebiete und der
              {H}olomorphiegebiete im {R}aum von {$n$} komplexen
              {V}er\"{a}nderlichen.
	\newblock {\em Math. Ann.}, 128:63--91 (1954)
	
	\bibitem{CWW} B.~Chen, J.~Wu, X.~Wang,
	\newblock Ohsawa-{T}akegoshi type theorem and extension of plurisubharmonic functions.
	\newblock {\em Math. Ann.}, 362:305--319 (2015)


\bibitem{Dem-82} J.~P.~Demailly,
\newblock Estimations {$L^{2}$} pour l'op\'{e}rateur {$\bar \partial $}
              d'un fibr\'{e} vectoriel holomorphe semi-positif au-dessus d'une
              vari\'{e}t\'{e} k\"{a}hl\'{e}rienne compl\`ete.
             \newblock {\em Ann. Sci. \'Ecole Norm. Sup. (4)},  15(3):457--511 (1982)

	\bibitem{Demaily} J.~P.~Demailly,
		 	\newblock \emph{Complex analytic and differential geometry}.
		 	\newblock Electric book, available on the author's homepage.
			
	
			
	\bibitem{DNW21} F.~Deng, J.~Ning, Z.~Wang,
	\newblock Characterizations of plurisubharmonic functions.
	\newblock {\em Sci China Math}, 64:1959--1970 (2021)

    \bibitem{DNWZ22}
	F.~Deng, J.~Ning, Z.~Wang, X.~Zhou,
	\newblock Positivity of holomorphic vector bundles in terms of $L^p$ estimates for $\dbar$.
    \newblock {\em Math. Ann.}, 385:575--607 (2023)

    \bibitem{DWZZ} F.~Deng, Z.~Wang, L.~Zhang, X.~Zhou,
    \newblock New characterizations of plurisubharmonic functions and positivity of direct image sheaves. {\em arXiv:1809.10371},
(To appear in {\em Amer. J.  Math.}).

    \bibitem{Deng-Zhang}
    F.~Deng, X.~Zhang,
	\newblock Characterizations of Curvature positivity of
	Riemannian vector bundles and convexity or pseudoconvexity of bounded domains in $\mathbb{R}^n$ or $\mathbb{C}^n$ in terms of $L^2$-estimate of $d$ of $\bar\partial$ equation.
	\newblock {\em J. Funct. Anal.}, 281(9):109184, 30.pp (2021)

    \bibitem{D-F} K.~Diederich, J.~E.~Forn\ae ss,
    \newblock Thin complements of complete K\"ahler domains.
    \newblock {\em Math. Ann.}, 259(3):331--341 (1982)	


	\bibitem{Diederich-Pflug} K.~Diederich, P.~Pflug,
	\newblock \"Uber Gebiete mit vollst\"andiger K\"ahlermetrik.
	\newblock {\em Math. Ann.}, 257(2):191--198 (1981)


	\bibitem{Docquier-Grauert} F.~Docquier, H.~Grauert,
	\newblock Levisches Problem und Rungescher Satz f\"ur Teilgebiete Steinscher Mannigfaltigkeiten.
	\newblock{\em Math. Ann.}, 140:94--123 (1960)
	


    \bibitem{Grauert}
   H.~Grauert, \newblock Charakterisierung der Holomorphiegebiete durch die vollständige Kählersche Metrik. \newblock{\em Math. Ann.} 131:38--75 (1956)



	\bibitem{Hor65}
	L.~H\"ormander,
	\newblock {$L^{2}$} estimates and existence theorems for the {$\bar \partial
		$}\ operator.
	\newblock {\em Acta Math.}, 113:89--152 (1965)

\bibitem{In20}
    T.~Inayama,
    \newblock Nakano positivity of singular Hermitian metrics and vanishing
    theorems of {D}emailly-{N}adel-{N}akano type.
    \newblock {\em Algebr. Geom.}, 9(1):69--92 (2022)

\bibitem{kobayashi}
S. Kobayashi,	
Differential geometry of complex vector bundles,
Publications of the Mathematical Society of Japan, 15. Kan\^{o} Memorial Lectures, 5. Princeton University Press, Princeton (1987)

	\bibitem{Kohn63} J.~J.~Kohn,
	\newblock Harmonic integrals on strongly pseudo-convex manifolds. I. II.
	\newblock {\em Ann. of Math. (2)}, 78:112--148 (1963); 79:450--472 (1964)

    \bibitem{Laufer} H.~B.~Laufer,
    \newblock On sheaf cohomology and envelopes of holomorphy.
   \newblock {\em Ann. of Math. (2)}, 84:102--118 (1966)

    \bibitem{LYZ21}
    Z.~Liu, B.~Xiao, H.~Yang, X.~Zhou,
	 \newblock Multiplier submodule sheaves and a problem of Lempert.
	\newblock \url{arXiv:2111.13452v2}.

	\bibitem{Norguet}  F.~Norguet,
	\newblock Sur les domaines d'holomorphie des fonctions uniformes de plusieurs variables complexes.
	\newblock {\em Bull. Soc. Math. France}, 82:137--159 (1954)
	
	
		\bibitem{Oh-80} T.~Ohsawa,
         \newblock On complete K\"ahler domains with $C^1$-boundary.
         \newblock {\em Publ. Res. Inst. Math. Sci.}, 16(3):929--940 (1980)


         \bibitem{Oh-02} T.~Ohsawa,
         \newblock A precise $L^2$ division theorem.
         \newblock {\em Complex geometry (Göttingen, 2000), Springer, Berlin}, 185--191 (2002)

	\bibitem{Oka} K.~Oka,
	\newblock Sur les fonctions de plusieurs variables. IX. Domaines finis sans point critique int\'erieur.
	\newblock {\em Jpn. J. Math.}, 23:97--155 (1953)
	
     \bibitem{PT18}
    M.~P\u{a}un, S.~Takayama,
    \newblock Positivity of twisted relative pluricanonical bundles and their
    direct images.
    \newblock {\em J. Algebraic Geom.}, 27(2):211--272 (2018)

    \bibitem{PW} P.~Pflug, W.~Zwonek,
    \newblock $L^2_{h}$-domains of holomorphy and the Bergman kernel.
    \newblock {\em Studia Math.}, 151: 99--108 (2002)

    \bibitem{Rau15}
    H.~Raufi,
    \newblock Singular Hermitian metrics on holomorphic vector bundles.
    \newblock {\em Ark. Mat.}, 53(2):359--382 (2015)

    \bibitem{Shaw}  C.~Laurent-Thi\'{e}baut, M.~Shaw,
\newblock On the {H}ausdorff property of some {D}olbeault cohomology
              groups.
              \newblock {\em Math. Z.}, 274(3-4):1165--1176 (2013)

	\bibitem{ShawMeichi} M.~Shaw,
    \newblock Global solvability and regularity for $\dbar$ on an annulus between two weakly pseudoconvex domains.
    \newblock {\em Trans. Amer. Math. Soc.}, 291(1):255--267 (1985)


	\bibitem{Skoda} H.~Skoda,
	\newblock Application des techniques $L^2$
	\`a la th\`eorie des id\`eaux d'une alg\`ebre de
	fonctions holomorphes avec poids.
	\newblock{\em Ann. Sci. \'Ecole Norm. Sup. (4)}
	 5:545--579 (1972)
	
	 \bibitem{Skoda-2} H.~Skoda,
       \newblock Morphismes surjectifs de fibr\'es vectoriels semi-positifs.
       \newblock {\em Ann. Sci. \'Ecole Norm. Sup. (4)} 11(4):577--611 (1978)

\bibitem{Varolin} D.~Varolin,
\newblock Division theorems and twisted complexes.
\newblock {\em Math. Z.}, 259(1):1--20 (2008)

\bibitem{Yasuoka}  T.~Yasuoka,
\newblock A domain exhausted by complete {K}\"{a}hler domains.
\newblock {\em Kobe J. Math.}, 2(1):57--64 (1985)

	\end{thebibliography}
\end{document}